\documentclass[12pt,a4paper]{article}%
\usepackage[utf8]{inputenc}
\usepackage{hyperref}
\usepackage{amsmath}
\usepackage{amsfonts}
\usepackage{amssymb}
\usepackage{xcolor}
\usepackage{graphicx}%
\usepackage{tikz}%
\usetikzlibrary{calc}%
\usetikzlibrary{decorations.pathreplacing}%
\usepackage{xspace}%
\usepackage{lineno}%
\providecommand{\U}[1]{\protect\rule{.1in}{.1in}}
\newtheorem{theorem}{Theorem}

\newtheorem{corollary}[theorem]{Corollary}

\newtheorem{lemma}[theorem]{Lemma}

\newtheorem{observation}[theorem]{Observation}

\newtheorem{proposition}[theorem]{Proposition}
\newtheorem{remark}[theorem]{Remark}

\newenvironment{proof}[1][Proof]{\noindent\textbf{#1.} }{\ \hfill \rule{0.5em}{0.5em}\bigskip}
\graphicspath{{Slike/}}

\newcommand{\MOP}{\textup{MOP}\xspace}
\newcommand{\MOPs}{\textup{MOPs}\xspace}
\begin{document}

\title{An improved upper bound for the fair domination number of maximal
outerplanar graphs}
\author{Yair Caro$^{1}$, Riste \v{S}krekovski$^{2}$ \\[0.3cm]
\parbox{0.92\textwidth}{\centering\small $^{1}$ \textit{Department of Mathematics and Physics, University of Haifa-Oranim, Tivon 36006, Israel}\\[0.1cm] $^{2}$ \textit{Faculty of
Mathematics and Physics, University of Ljubljana; Rudolfovo --
Science and Technology Centre Novo Mesto; Faculty of Information
Studies, University of Novo Mesto, Slovenia}}
}
\date{}
\maketitle

\begin{abstract}
A dominating set $D$ of a graph $G$ is a \emph{fair dominating set} if every
two vertices outside $D$ have the same number of neighbors in $D$, and the
\emph{fair domination number} $\mathrm{fd}(G)$ is the minimum cardinality of
such a set. Caro, Hansberg and Henning, who introduced this parameter, proved
that $\mathrm{fd}(G)<17n/19$ for every maximal outerplanar graph $G$ of order
$n\geq3$, and asked whether this bound is asymptotically best possible. We show
that it is not the case by proving $\mathrm{fd}(G)\leq(7n-3)/8<7n/8$ for every maximal
outerplanar graph $G$ of order $n\geq3$, and we exhibit an infinite family of
maximal outerplanar graphs with $\mathrm{fd}(G)/n\rightarrow7/9$, so that the
best asymptotic constant lies between $7/9$ and $7/8$.
\end{abstract}



\section{Introduction}

For notation and graph theory terminology not defined here we follow
\cite{Harary}, \cite{CaroHansbergHenning} and \cite{West}. Let $G=(V,E)$ be a graph of order $n=|V|$ and let
$N(v)$ denote the open neighborhood of a vertex $v$, $d(v)=|N(v)|$ its degree.
A set $D\subseteq V$ is a \emph{dominating set} if every vertex outside $D$
has a neighbor in $D$.

Caro, Hansberg and Henning \cite{CaroHansbergHenning} introduced the following
variant of domination. A dominating set $D$ of $G$ is a \emph{fair dominating
set} (an FD-set) if $\left\vert N(u)\cap D\right\vert =\left\vert N(v)\cap
D\right\vert $ for every two vertices $u,v\in V\setminus D$; the \emph{fair
domination number} $\mathrm{fd}(G)$ is the minimum cardinality of an FD-set of
$G$. Dually, a set $Q\subseteq V$ is an \emph{out-regular set} (an OR-set) if
$\left\vert N(u)\setminus Q\right\vert =\left\vert N(v)\setminus Q\right\vert
>0$ for every $u,v\in Q$; this common value is called the \emph{out-degree}
of $Q$, and $\xi_{or}(G)$ denotes the maximum cardinality of
an OR-set of $G$. Among other results, \cite{CaroHansbergHenning} establishes
the identity
\[
\mathrm{fd}(G)+\xi_{or}(G)=n
\]
for every graph $G$ of order $n\geq2$, and uses it to bound $\mathrm{fd}(G)$
for several graph classes. In particular, if $G$ is a \emph{maximal
outerplanar graph} (a \MOP, that is, a triangulation of a convex polygon) of
order $n\geq3$, then
\[
\mathrm{fd}(G)<\frac{17n}{19},
\]
and the authors remark that they do not know whether this bound is
asymptotically best possible, in the sense that $\mathrm{fd}(G)/n\rightarrow
17/19$ as $n\rightarrow\infty$; improving it is listed as an open problem in
\cite{CaroHansbergHenning}.

Let us first mention that fair domination has received much attention, with
nearly $100$ papers citing \cite{CaroHansbergHenning} (as can be seen in
Google Scholar, September 2026), and yet after more than a decade the bound
$17n/19$ has stood without improvement. An explicit attempt
to improve it has been made in \cite{HajianJafariRad}; however, the
improvement obtained there holds only for outerplanar graphs with
$e(G)\leq91n/76$, which is far from maximal outerplanar graphs.

The proof of the bound $17n/19$ is based on the notion of a \emph{regular
independent set}: $\alpha_{reg}(G)$ denotes the maximum cardinality of an
independent set all of whose vertices have equal degree in $G$. This
parameter was introduced by Albertson and Boutin \cite{AlbertsonBoutin} and
further developed, together with its degree-restricted variants, by Caro,
Hansberg and Pepper \cite{CaroHansbergPepper}, who, among other graph
classes, consider $2$-trees and hence in particular \MOPs. The bound
$\alpha_{reg}(G)\geq\tfrac{2}{19}(n+3)$ for every \MOP $G$ is exactly what
the proof of $17n/19$ in \cite{CaroHansbergHenning} implicitly establishes,
as is recorded in \cite[Table~1]{CaroHansbergPepper}; indeed
$n-\tfrac{2}{19}(n+3)=(17n-6)/19<17n/19$.

The first author of this paper suggested that a further improvement of this upper bound
should come from a more general argument than merely squeezing further the
idea of a regular independent set, namely to reconsider large out-regular
subgraphs of a \MOP, rather than restricting attention to regular independent
sets. We revisit the proof of the bound $17n/19$ and find that the mechanism
it already uses -- an out-regular subset of a single degree class -- can
indeed be pushed considerably further once one drops the requirement that
the set be independent and instead only asks that it induce a regular graph.
In this paper we show that this refinement, applied to the vertices of
degree $2$ and $3$ of a \MOP simultaneously and extending the result and
technique of \cite{CaroHansbergHenning}, gives
$\mathrm{fd}(G)\leq(7n-3)/8<7n/8$, and we further exhibit, in
Section~\ref{Sec_lower}, an infinite family of \MOPs with
$\mathrm{fd}(G)/n\rightarrow7/9$; so the best asymptotic constant lies between
$7/9$ and $7/8$, and we do not know its value. We close with some remarks on where the gap between these two
constants comes from.

\section{Preliminaries}\label{Sec_prelim}

We fix a \MOP $G$ of order $n\geq3$ together with its outer cycle
$C=v_{0}v_{1}\cdots v_{n-1}v_{0}$; the edges of $C$ are the \emph{outer edges}
of $G$, all remaining edges are \emph{chords}. For $i\geq2$ we write
$V_{i}=\{v\in V(G):d(v)=i\}$ and $n_{i}=\left\vert V_{i}\right\vert $.

We will use the following well known structural facts about \MOPs, all of
which are folklore or immediate from the fact that a \MOP is precisely a
triangulated polygon; see \cite{CaroHansbergHenning} and the references
therein.

\begin{observation}
\label{Obs_F1}Let $G$ be a \MOP of order $n\geq3$. Then $\left\vert
E(G)\right\vert =2n-3$, and hence $\sum_{v\in V(G)}d(v)=4n-6$; moreover
$\delta(G)=2$, and $G$ is both $2$-degenerate and $3$-colorable.
\end{observation}

\begin{observation}
[\cite{LaskarMulderNovick}]\label{Obs_F2}For every vertex $v$ of a \MOP $G$,
the set $N(v)$ induces a path in $G$.
\end{observation}

The next refinement of Observation \ref{Obs_F2} identifies the ends of that
path.

\begin{observation}
\label{Obs_F3}Let $v$ be a vertex of a \MOP $G$ and let $w,w^{\prime}$ be the
two neighbors of $v$ on $C$. Then $w,w^{\prime}$ are precisely the two ends of
the path induced by $N(v)$.
\end{observation}

\begin{proof}
Let $d=d(v)$ and let $w=w_{0},w_{1},\ldots,w_{d-1}=w^{\prime}$ be the
neighbors of $v$ listed in the cyclic order in which they occur around $C$.
The triangles of $G$ incident with $v$ form a fan on $v$, so $w_{j}w_{j+1}\in
E(G)$ for $0\leq j\leq d-2$. These are already $d-1$ edges inside $N(v)$, and
by Observation \ref{Obs_F2} the graph $G[N(v)]$ is a path, which has exactly
$d-1$ edges. Hence $G[N(v)]$ is exactly the path $w_{0}w_{1}\cdots w_{d-1}$,
whose ends are $w=w_{0}$ and $w^{\prime}=w_{d-1}$.
\end{proof}

\begin{observation}
\label{Obs_F4}If $G$ is a \MOP of order $n\geq4$, then $n_{2}\geq2$ and
$V_{2}$ is independent.
\end{observation}

We will make repeated use of the identity of \cite{CaroHansbergHenning}
recalled in the introduction.

\begin{proposition}
[\cite{CaroHansbergHenning}]\label{Prop_20}For every graph $G$ of order
$n\geq2$, $\mathrm{fd}(G)+\xi_{or}(G)=n$.
\end{proposition}

The following observation is the starting point of this paper: it isolates
exactly what the proof of the bound $17n/19$ in \cite{CaroHansbergHenning}
uses, and what it leaves for further improvements.

\begin{observation}
\label{Obs_reg}Let $G$ be a \MOP of order $n\geq4$. Let $i\geq2$, let
$c\in\{0,1,2\}$ and let $Q\subseteq V_{i}$ be a non-empty set such that
$G[Q]$ is $c$-regular. Then $Q$ is an OR-set of out-degree $i-c$. (The
restriction $c\leq2$ is no loss of generality: being $2$-degenerate by
Observation \ref{Obs_F1}, a \MOP has no $3$-regular induced subgraph.)
\end{observation}

\begin{proof}
If $i=2$, then $V_{2}$ is independent (Observation \ref{Obs_F4}), so $Q$ is
$0$-regular, $c=0$, and $Q$ is an OR-set of out-degree $2$. For $i\geq3$,
$\left\vert N(u)\setminus Q\right\vert =d(u)-\left\vert N(u)\cap Q\right\vert
=i-c\geq i-2\geq1$ for every $u\in Q$.
\end{proof}

In \cite{CaroHansbergHenning}, only $c=0$ is used, that is, only independent
subsets of a single degree class; Observation \ref{Obs_reg} shows that this
is an unnecessary restriction, and the rest of the paper exploits the
freedom $c\in\{1,2\}$ for the classes $V_{2}$ and $V_{3}$.

\section{\texorpdfstring{The structure of $V_{2}\cup V_{3}$ in a \MOP}%
{The structure of V2 u V3 in a MOP}}\label{Sec_struct}

This section is devoted to a better understanding of the structure of the
graph induced by $V_{2}\cup V_{3}$ in a \MOP, and in particular to
strengthening Lemma 21 of \cite{CaroHansbergHenning}, which was crucial in
obtaining the upper bound $17n/19$.

\begin{lemma}
\label{Lemma_V2V3}Let $G$ be a \MOP of order $n\geq5$ and let $uv\in E(G)$
with $u,v\in V_{2}\cup V_{3}$. Then $uv$ is an outer edge.
\end{lemma}

\begin{proof}
If $d(u)=2$, then by Observation \ref{Obs_F3} both neighbors of $u$ are its
outer neighbors, so every edge at $u$, in particular $uv$, is an outer edge.
So, by the symmetry of the roles of $u$ and $v$, we may assume $d(u)=d(v)=3$
and, to the contrary, that $uv$ is a chord. Let
$N(u)=\{w_{0},w_{1},w_{2}\}$, where by Observation \ref{Obs_F3} the vertices
$w_{0},w_{2}$ are the outer neighbors of $u$ and $G[N(u)]$ is the path
$w_{0}w_{1}w_{2}$; in particular $w_{0}w_{2}\not \in E(G)$. Since $uv$ is a
chord, $v=w_{1}$. Now $v$ is adjacent to $w_{0},u,w_{2}$ and $d(v)=3$, so
$N(v)=\{w_{0},u,w_{2}\}$. Inside $N(v)$ we have the edges $w_{0}u$ and
$uw_{2}$, and not the edge $w_{0}w_{2}$, so $G[N(v)]$ is the path
$w_{0}\,u\,w_{2}$, whose ends $w_{0},w_{2}$ are, by Observation
\ref{Obs_F3}, the outer neighbors of $v$. Consequently $uw_{0},uw_{2}%
,vw_{0},vw_{2}$ are all outer edges, so $C$ is the $4$-cycle
$u\,w_{0}\,v\,w_{2}$ and $n=4$, a contradiction.
\end{proof}

We call the connected components of $G[V_{2}\cup V_{3}]$ the \emph{runs} of
$G$.

\begin{corollary}
\label{Cor_runs}Let $G$ be a \MOP of order $n\geq5$. Then $V_{2}\cup
V_{3}\not =V(G)$, every run of $G$ consists of consecutive vertices of $C$,
no edge of $G$ joins two distinct runs, and $G[V_{3}]$ is bipartite.
\end{corollary}

\begin{proof}
If $V_{2}\cup V_{3}=V(G)$, then $n_{2}+n_{3}=n$ and, by Observation
\ref{Obs_F1}, $2n_{2}+3n_{3}=4n-6$; together these give $n_{2}=6-n\leq1$ for
$n\geq5$, contradicting Observation \ref{Obs_F4}. By Lemma \ref{Lemma_V2V3},
$G[V_{2}\cup V_{3}]$ is a subgraph of $C$, and a proper subgraph of a cycle is
a disjoint union of paths on consecutive vertices; two vertices lying on
different such paths are non-adjacent on $C$, hence, by Lemma
\ref{Lemma_V2V3} once more, non-adjacent in $G$. Since each run is a path,
$G[V_{3}]$, being a union of subpaths of runs, is bipartite.
\end{proof}

Notice that Corollary \ref{Cor_runs} already reproves Lemma 21 of \cite
{CaroHansbergHenning} in the stronger form above.

\begin{lemma}
\label{Lemma_endrun}Let $G$ be a \MOP of order $n\geq5$. Then every $x\in
V_{2}$ has at most one neighbor in $V_{3}$; in particular, $x$ is an end
vertex of its run.
\end{lemma}

\begin{proof}
Let $N(x)=\{a,b\}$. By Observation \ref{Obs_F2}, $ab\in E(G)$. Suppose that
$a,b\in V_{3}$; then each of them has exactly one more neighbor.
Suppose $a$ has neighbor $u$ and $b$ has neighbor $v$ with $u\neq v$. Since $x$ has
degree 2, by Observation \ref{Obs_F3}, vertices $a$ and $b$ are the two outer-neighbors of $x$.
Thus $ab$ is a chord (not an outer edge), contradicting Lemma \ref{Lemma_V2V3}.
If they share a common neighbor $z\neq x$, then as $n\geq5$, $z$ is a cut point,
contradicting the fact that $G$, as a \MOP, is $2$-connected.
The last statement follows since $V_{2}$ is independent (Observation \ref{Obs_F4}).
\end{proof}

\begin{lemma}
\label{Lemma_fan}Let $G$ be a \MOP of order $n\geq5$ and let $R$ be a run of
$G$ containing two vertices of $V_{2}$. Then $G$ is the fan $F_{n}%
=P_{n-1}+K_{1}$ (one hub vertex joined to all vertices of a path on the
remaining $n-1$ vertices), and
$R=V(G)\setminus\{w\}$, where $w$ is the hub.
\end{lemma}

\begin{proof}
By Lemma \ref{Lemma_endrun}, $R=x\,u_{1}u_{2}\cdots u_{r}\,y$ with
$x,y\in V_{2}$, $u_{1},\ldots,u_{r}\in V_{3}$, as $R$ is a run, and $r\geq1$,
as by Observation \ref{Obs_F4}, $x$ and $y$ are independent. Let $w$ be the
outer neighbor of $x$ not in $R$; then $d(w)\geq4$ and, by Observation
\ref{Obs_F2}, $wu_{1}\in E(G)$. We show by induction that $wu_{i}\in E(G)$
for every $1\leq i\leq r$; the case $i=1$ is done. Let $1\leq i<r$ and assume
$wu_{i}\in E(G)$. Writing $u_{0}=x$, the three vertices $u_{i-1},u_{i+1},w$
are distinct neighbors of $u_{i}$ and $d(u_{i})=3$, so $N(u_{i}%
)=\{u_{i-1},u_{i+1},w\}$. By Observation \ref{Obs_F3}, the ends of the path
$G[N(u_{i})]$ are the outer neighbors $u_{i-1},u_{i+1}$ of $u_{i}$, so that
path is $u_{i-1}\,w\,u_{i+1}$; in particular $wu_{i+1}\in E(G)$, which
completes the induction. Hence $N(u_{r})=\{u_{r-1},y,w\}$. Let $z$ be the
outer neighbor of $y$ not in $R$; by Observation
\ref{Obs_F2}, $zu_{r}\in E(G)$, so $z\in\{u_{r-1},w\}$. Since $z$ is an outer
neighbor of $y$ while $u_{r-1}$ is not (for $r\geq2$ it lies at distance $2$
from $y$ along $C$; for $r=1$, $u_{0}=x$, and $z=x$ would force $n=3$), we
obtain $z=w$. So $w$ is adjacent to every vertex of $R$, and its two outer
neighbors are $x$ and $y$; thus $C=w\,x\,u_{1}\cdots u_{r}\,y$ and $G$ is the
fan with hub $w$.
\end{proof}

\section{\texorpdfstring{A single out-regular set for $V_{2}$ and $V_{3}$ together}%
{A single out-regular set for V2 and V3 together}}

\begin{theorem}
\label{Tm_half}If $G$ is a \MOP of order $n\geq5$, then $\xi_{or}%
(G)\geq\frac{1}{2}(n_{2}+n_{3})$.
\end{theorem}

\begin{proof}
If $G$ is a fan with hub $w$, then $Q=V(G)\setminus\{w\}$ is an OR-set of
out-degree $1$, since every vertex of the path on $Q$ has exactly one
neighbor outside $Q$, namely $w$; since $w$ has degree $n-1\geq4$ for
$n\geq5$, Corollary \ref{Cor_runs} gives $n_{2}+n_{3}\leq n-1$, so
$\xi_{or}(G)\geq\left\vert Q\right\vert =n-1\geq n_{2}+n_{3}$ and we are done.
(In fact $\xi_{or}(G)=n-1$ here, since $Q=V(G)$ has out-degree $0$ and is
therefore not an OR-set.) So we may assume that $G$
is not a fan, and, by Lemma
\ref{Lemma_fan}, that every run $R$ of $G$ contains at most one vertex of
$V_{2}$.

Fix a run $R$, put $h=\left\vert R\right\vert $, $t=\left\vert R\cap
V_{2}\right\vert \in\{0,1\}$ and $r=\left\vert R\cap V_{3}\right\vert
=h-t$, and define
\[
v_{3}(R)=\max\{\left\vert Q\right\vert :Q\subseteq R\cap V_{3},\ Q\text{
independent}\},
\]
\[
v_{2}(R)=\max\{\left\vert Q\right\vert :Q\subseteq R,\ \left\vert N(u)\cap
Q\right\vert =0\ \forall u\in Q\cap V_{2},\ \left\vert N(u)\cap Q\right\vert
=1\ \forall u\in Q\cap V_{3}\}.
\]

\medskip

\noindent\textbf{Claim A.} \emph{For every run }$R$\emph{, }$v_{2}%
(R)+v_{3}(R)\geq h$.\smallskip

\noindent The mechanism behind $v_{2}(R)$ is the following, and it is the
whole idea of this section: the vertex set of an \emph{induced matching}
inside $R\cap V_{3}$ is a set in which every vertex has exactly one neighbor,
hence out-degree $3-1=2$, which is precisely the out-degree of a vertex of
$V_{2}$ having no neighbor in the set. This is what allows $V_{2}$ and
$V_{3}$ to be packed into one and the same OR-set.

Since $R\cap V_{3}$ induces the path $P_{r}$, we have $v_{3}%
(R)=\left\lceil r/2\right\rceil $, so it suffices to show $v_{2}%
(R)\geq\left\lfloor r/2\right\rfloor +t$. Recall that the maximum number of
edges of an induced matching of a path on $\ell$ vertices is $\left\lceil
(\ell-1)/3\right\rceil $ for $\ell\geq1$ (and $0$ for $\ell\leq1$); the
vertex set of an induced matching inside $R\cap V_{3}$ is admissible for
$v_{2}(R)$, since each of its vertices then has exactly one neighbor in it.

If $t=0$, this already gives $v_{2}(R)\geq2\left\lceil (r-1)/3\right\rceil $,
and $2\left\lceil (r-1)/3\right\rceil \geq\left\lfloor r/2\right\rfloor $
holds for $r\leq3$ by inspection and for $r\geq4$ because $2(r-1)/3\geq
r/2$ once $r\geq4$.

If $t=1$ and $r=0$, the run consists of a single degree-$2$ vertex $x$, and
$Q=\{x\}$ is admissible for $v_{2}(R)$, so $v_{2}(R)\geq1=\left\lfloor
0/2\right\rfloor +1$, as required. If $t=1$ and $r\geq1$, write
$R=x\,u_{1}\cdots u_{r}$ with $x\in V_{2}$. Taking an induced
matching of the subpath $u_{2}\cdots u_{r}$ together with $x$ (admissible,
since then $u_{1}\not \in Q$, so $x$ has no neighbor in $Q$) gives
$v_{2}(R)\geq2\left\lceil (r-2)/3\right\rceil +1$; taking instead an induced
matching of the whole subpath $u_{1}\cdots u_{r}$ gives $v_{2}(R)\geq
2\left\lceil (r-1)/3\right\rceil $. Checking $r\leq7$ directly shows that
the larger of these two already reaches $\left\lfloor r/2\right\rfloor +1$
in every case, and for $r\geq8$ the first bound alone suffices, since
$4(r-2)\geq3r$ once $r\geq8$. This establishes Claim A.

\medskip

Summing Claim A over all runs $R$ and using Corollary \ref{Cor_runs} gives
\[
\sum_{R}v_{2}(R)+\sum_{R}v_{3}(R)\geq\sum_{R}\left\vert R\right\vert
=n_{2}+n_{3},
\]
so $\max\{\sum_{R}v_{2}(R),\sum_{R}v_{3}(R)\}\geq\frac{1}{2}(n_{2}+n_{3})$.
It remains to observe that both sums are cardinalities of OR-sets of $G$. Let
$Q$ be the union, over all runs, of the corresponding optimal sets. By
Corollary \ref{Cor_runs} there is no edge between two distinct runs, so for
every $u\in Q$ the neighbors of $u$ inside $Q$ are exactly its chosen
neighbors within its own run. In either case $\left\vert Q\right\vert
\geq\frac{1}{2}(n_{2}+n_{3})\geq1$ by Observation \ref{Obs_F4}, so $Q$ is
non-empty. Hence, if $Q$ realizes $\sum_{R}v_{3}(R)$, then $Q\subseteq V_{3}$ is
independent by the definition of $v_{3}(R)$ together with Corollary
\ref{Cor_runs} (independent within each run, and no edges between distinct
runs), so $Q$ is an OR-set of out-degree $3$ by Observation \ref{Obs_reg};
and if $Q$ realizes $\sum_{R}v_{2}(R)$, then, by the definition of $v_{2}(R)$
together with Corollary \ref{Cor_runs} once more, $\left\vert N(u)\setminus
Q\right\vert =2$ for $u\in Q\cap V_{2}$ and $\left\vert N(u)\setminus
Q\right\vert =3-1=2$ for $u\in Q\cap V_{3}$, so $Q$ is an OR-set of
out-degree $2$. In either case $\xi_{or}(G)\geq\left\vert Q\right\vert
\geq\frac{1}{2}(n_{2}+n_{3})$.
\end{proof}

\begin{remark}
\label{Rem_half_sharp}Theorem \ref{Tm_half} is sharp: computation over all
\MOPs of order $n\leq16$ gives $\min\,\xi_{or}(G)/(n_{2}+n_{3})=1/2$, attained
from $n=12$ on (and for the family $H_{m}$ of Section \ref{Sec_lower} the
ratio is $(2m+2)/(4m+3)$, which tends to $1/2$). Moreover, since $\xi
_{or}(G)$ is an integer, Theorem \ref{Tm_half} gives $\xi_{or}%
(G)\geq\left\lceil (n_{2}+n_{3})/2\right\rceil $, and the family $H_{m}$
attains this integral form with equality for \emph{every} $m\geq2$, as
$\left\lceil (4m+3)/2\right\rceil =2m+2=\xi_{or}(H_{m})$ by Theorem
\ref{Tm_strip}.
\end{remark}

%----------------------------------------------------------------------------------------------------
\section{The improved bound}\label{Sec_upper}

\begin{theorem}
\label{Tm_main}If $G$ is a \MOP of order $n\geq3$, then
\[
\mathrm{fd}(G)\leq\frac{7n-3}{8}<\frac{7n}{8}.
\]
\end{theorem}

\begin{proof}
For $n\in\{3,4\}$, $\mathrm{fd}(G)=1$ and the claim holds. So let $n\geq5$
and put $\xi=\xi_{or}(G)$. We record three lower bounds on $\xi$.

By Observation \ref{Obs_F4}, $V_{2}$ is a non-empty independent set, hence an
OR-set of out-degree $2$ by Observation \ref{Obs_reg}, so
\[
n_{2}\leq\xi.\tag{i}
\]
By Theorem \ref{Tm_half},
\[
n_{2}+n_{3}\leq2\xi.\tag{ii}
\]
For every $i\geq4$ with $n_{i}\geq1$, since $G$ is $3$-colorable
(Observation \ref{Obs_F1}), $V_{i}$ contains an independent set of size at
least $n_{i}/3$, which is an OR-set of out-degree $i\geq1$ by Observation
\ref{Obs_reg}, so
\[
n_{i}\leq3\xi\qquad(i\geq4).\tag{iii}
\]

Counting vertices and degrees, $\sum_{i\geq2}n_{i}=n$ and $\sum_{i\geq
2}i\,n_{i}=4n-6$ by Observation \ref{Obs_F1}, so
\[
\sum_{i\geq2}(6-i)\,n_{i}=6n-(4n-6)=2n+6,
\]
that is,
\[
4n_{2}+3n_{3}+2n_{4}+n_{5}=2n+6+\sum_{i\geq7}(i-6)n_{i}\geq2n+6.\tag{1}\label{For_count}
\]
On the other hand, using $4n_{2}+3n_{3}+2n_{4}+n_{5}=n_{2}+3(n_{2}%
+n_{3})+2n_{4}+n_{5}$ together with (i)--(iii),
\[
4n_{2}+3n_{3}+2n_{4}+n_{5}\leq\xi+6\xi+6\xi+3\xi=16\,\xi.
\]
Combining this with (\ref{For_count}) gives $16\xi\geq2n+6$, that is,
$\xi\geq(n+3)/8$. By Proposition \ref{Prop_20}, $\mathrm{fd}(G)=n-\xi\leq
n-(n+3)/8=(7n-3)/8$.
\end{proof}

Since $7/8=0.875<17/19=0.8947\ldots$, Theorem \ref{Tm_main} improves the
bound of \cite{CaroHansbergHenning} and, in particular, shows that
$\mathrm{fd}(G)/n$ does \emph{not} tend to $17/19$ as $n\rightarrow\infty$.

\begin{remark}
\label{Rem_bottleneck}The proof of Theorem \ref{Tm_main} only sharpens the
counting argument of \cite{CaroHansbergHenning} for the classes $V_{2}$ and
$V_{3}$; bound (iii) for $i\in\{4,5\}$ is untouched, and it is the
bottleneck. If (iii) could be replaced by $n_{i}\leq\xi/\rho$ for $i\in
\{4,5\}$ and some $\rho>1/3$, the same two lines of computation would give
\[
\mathrm{fd}(G)\leq n-\frac{2n+6}{7+3/\rho}.\tag{2}\label{For_rho}
\]
For instance, (\ref{For_rho}) with $\rho=1/2$ gives $\mathrm{fd}%
(G)<11n/13=0.8461\ldots n$, and with $\rho=2/5$ it gives $\mathrm{fd}%
(G)<25n/29=0.8620\ldots n$. Since $V_{4}$ and
$V_{5}$ need not induce a union of subpaths of $C$ (chords between two
vertices of $V_{4}$ are possible, already for $n=6$), the method of Section
\ref{Sec_struct} does not carry over directly. The complementary question of
when (\ref{For_count}) is tight is taken up in Remark \ref{Rem_delta6}.
\end{remark}

%------------------
\section{A lower bound}\label{Sec_lower}

We now show that Theorem \ref{Tm_main} is not far from best possible, by
exhibiting a family $H_{m}$ of \MOPs with $\mathrm{fd}(H_{m})/n\rightarrow7/9$
(Theorem \ref{Tm_strip}). The family was found by the computer search
described in Remark \ref{Rem_search} below; it has bounded maximum degree,
$\Delta(H_{m})=7$, and it is obtained by repeatedly gluing a fixed gadget (a certain triangulation of the convex 11-gon) along a strip.

%------
\subsection{\texorpdfstring{Construction of the gadget and the family $H_m$}%
{Construction of the gadget and the family Hm}}
Let $\Gamma$ be the triangulation of the $11$-gon $w_{0}w_{1}\cdots w_{10}$
with the eight chords
\[
w_{1}w_{10},\quad w_{2}w_{4},\quad w_{2}w_{10},\quad w_{4}w_{8},\quad
w_{4}w_{9},\quad w_{4}w_{10},\quad w_{5}w_{8},\quad w_{6}w_{8}%
\]
(Figure \ref{Fig_Hm}(a)); the degrees of $w_{0},\ldots,w_{10}$ in $\Gamma$
are $2,3,4,2,6,3,3,2,5,3,5$. For $m\geq2$, let $H_{m}$ be the graph obtained
from $m$ copies $\Gamma^{0},\ldots,\Gamma^{m-1}$ of $\Gamma$, with vertices
$w_{j}^{i}$, by identifying, for $0\leq i\leq m-2$, the outer edge
$w_{5}^{i}w_{6}^{i}$ of $\Gamma^{i}$ with the outer edge $w_{0}^{i+1}%
w_{10}^{i+1}$ of $\Gamma^{i+1}$, so that
\[
w_{5}^{i}=w_{0}^{i+1}\qquad\text{and}\qquad w_{6}^{i}=w_{10}^{i+1}%
\qquad(0\leq i\leq m-2).
\]
Gluing two triangulated polygons along an outer edge of each gives a
triangulated polygon, so $H_{m}$ is a \MOP. Its vertex set is the disjoint
union of $\{w_{0}^{0},w_{10}^{0}\}$ and the $m$ \emph{blocks} $B_{i}%
=\{w_{1}^{i},\ldots,w_{9}^{i}\}$, so $n=9m+2$. (In Figure
\ref{Fig_Hm}(b) the vertices are numbered along the outer cycle:
$w_{0}^{i},\ldots,w_{5}^{i}$ are $5i,\ldots,5i+5$ and $w_{6}^{i},\ldots
,w_{9}^{i}$ are $n-4i-5,\ldots,n-4i-2$, so a block consists of five
consecutive vertices on one side of the strip and four on the other. Note that
$w_{0}^{0}$ and $w_{10}^{0}$, the two vertices lying in no block, are the ends
of the outer edge closing the strip at its left, and $w_{5}^{m-1}w_{6}^{m-1}$
closes it at its right.)

\begin{figure}[!ht]
\centering
%--- (a) the gadget Gamma, a triangulation of the convex 11-gon -------------
\begin{tikzpicture}[scale=0.88]
\coordinate (g0) at (0.000,2.000);
\coordinate (g1) at (1.081,1.683);
\coordinate (g2) at (1.819,0.831);
\coordinate (g3) at (1.980,-0.285);
\coordinate (g4) at (1.511,-1.310);
\coordinate (g5) at (0.563,-1.919);
\coordinate (g6) at (-0.563,-1.919);
\coordinate (g7) at (-1.511,-1.310);
\coordinate (g8) at (-1.980,-0.285);
\coordinate (g9) at (-1.819,0.831);
\coordinate (g10) at (-1.081,1.683);
\draw[thick] (g0) -- (g1) -- (g2) -- (g3) -- (g4) -- (g5) -- (g6) -- (g7) --
  (g8) -- (g9) -- (g10) -- cycle;
\foreach \a/\b in {1/10,2/4,2/10,4/8,4/9,4/10,5/8,6/8}
  { \draw[blue!70!black,thick] (g\a) -- (g\b); }
\draw[red!70!black,very thick] (g10) -- (g0);
\draw[red!70!black,very thick] (g5) -- (g6);
\foreach \i in {0,...,10} { \fill (g\i) circle (1.8pt); }
% vertex names, and (in gray) the degree of the vertex in Gamma
\foreach \i/\d in {0/2,1/3,2/4,3/2,4/6,5/3,6/3,7/2,8/5,9/3,10/5}
  { \node[font=\small] at ($1.21*(g\i)$) {$w_{\i}$};
    \node[font=\scriptsize,gray!45!black] at ($1.44*(g\i)$) {\d}; }
\node[font=\small] at (0,-3.35) {(a) the gadget $\Gamma$};
\end{tikzpicture}

\vspace{2mm}

%--- (b) H_3 drawn as the strip it is --------------------------------------
\begin{tikzpicture}[x=0.92cm,y=0.86cm]
% upper row (w_0^i ... w_5^i), left to right
\coordinate (v0) at (0.000,3.000);
\coordinate (v1) at (1.000,3.073);
\coordinate (v2) at (2.000,3.142);
\coordinate (v3) at (3.000,3.566);
\coordinate (v4) at (4.000,3.260);
\coordinate (v5) at (5.000,3.303);
\coordinate (v6) at (6.000,3.333);
\coordinate (v7) at (7.000,3.348);
\coordinate (v8) at (8.000,3.708);
\coordinate (v9) at (9.000,3.333);
\coordinate (v10) at (10.000,3.303);
\coordinate (v11) at (11.000,3.260);
\coordinate (v12) at (12.000,3.206);
\coordinate (v13) at (13.000,3.502);
\coordinate (v14) at (14.000,3.073);
\coordinate (v15) at (15.000,3.000);
% lower row (w_6^i ... w_9^i and w_10^0), right to left
\coordinate (v16) at (15.000,0.000);
\coordinate (v17) at (13.750,-0.451);
\coordinate (v18) at (12.500,-0.175);
\coordinate (v19) at (11.250,-0.247);
\coordinate (v20) at (10.000,-0.303);
\coordinate (v21) at (8.750,-0.698);
\coordinate (v22) at (7.500,-0.350);
\coordinate (v23) at (6.250,-0.338);
\coordinate (v24) at (5.000,-0.303);
\coordinate (v25) at (3.750,-0.607);
\coordinate (v26) at (2.500,-0.175);
\coordinate (v27) at (1.250,-0.091);
\coordinate (v28) at (0.000,0.000);
% the outer cycle 0 1 ... 28 0
\draw[thick] (v0) -- (v1) -- (v2) -- (v3) -- (v4) -- (v5) -- (v6) -- (v7) --
  (v8) -- (v9) -- (v10) -- (v11) -- (v12) -- (v13) -- (v14) -- (v15) -- (v16)
  -- (v17) -- (v18) -- (v19) -- (v20) -- (v21) -- (v22) -- (v23) -- (v24)
  -- (v25) -- (v26) -- (v27) -- (v28) -- cycle;
% the 24 chords of the three copies of Gamma
\foreach \a/\b in {1/28,2/28,4/26,4/27,4/28,5/26,
  6/24,7/24,9/22,9/23,9/24,10/22,
  11/20,12/20,14/18,14/19,14/20,15/18}
  { \draw[blue!70!black,thick] (v\a) -- (v\b); }
% the chords w_2^i w_4^i and w_6^i w_8^i, bent clear of the degree-2 vertex
% w_3^i resp. w_7^i that they skip over
\foreach \a/\b in {2/4,7/9,12/14,24/26,20/22,16/18}
  { \draw[blue!70!black,thick] (v\a) to[bend right=14] (v\b); }
% the two glued edges
\draw[red!70!black,very thick] (v5) -- (v24);
\draw[red!70!black,very thick] (v10) -- (v20);
\foreach \i in {0,...,28} { \fill (v\i) circle (2.0pt); }
% the eight vertices of Q_2
\foreach \i in {0,3,8,13,15,16,21,25}
  { \draw[black,line width=1.0pt] (v\i) circle (5pt); }
% labels: upper row above, lower row below
\foreach \i in {0,...,15} { \node[font=\scriptsize,inner sep=1pt,above=8.5pt] at (v\i) {\i}; }
\foreach \i in {16,...,28} { \node[font=\scriptsize,inner sep=1pt,below=8.5pt] at (v\i) {\i}; }
% braces marking the three blocks
\foreach \s/\e/\lb in {1/5/0,6/10/1,11/15/2}
  { \draw[gray!55!black,decorate,decoration={brace,amplitude=4pt,mirror,raise=2pt}]
      (\s,-1.32) -- (\e,-1.32);
    \node[font=\small,gray!45!black] at ({(\s+\e)/2},-2.00) {$B_{\lb}$}; }
\node[font=\small] at (7.5,-2.85) {(b) the \MOP $H_{3}$, drawn as a strip};
\end{tikzpicture}
\caption{(a) The gadget $\Gamma$, a triangulation of the convex $11$-gon: blue
edges are its eight chords, the red outer edges $w_{10}w_{0}$ and $w_{5}w_{6}$
are those along which consecutive copies are glued, and the small gray number
at each vertex is its degree in $\Gamma$. (b) The \MOP $H_{3}$ ($m=3$,
$n=29$), drawn as the strip that it is: each of the three copies of $\Gamma$
contributes five vertices to the upper row and four to the lower one, and the
braces mark the blocks $B_{0},B_{1},B_{2}$. Vertices carry their number on the
outer cycle (black edges), under the numbering rule stated above; blue edges
are the chords of the three copies, and the two vertical red edges are the
glued ones, $\{5,24\}=w_{5}^{0}w_{6}^{0}=w_{0}^{1}w_{10}^{1}$ and
$\{10,20\}=w_{5}^{1}w_{6}^{1}=w_{0}^{2}w_{10}^{2}$. The eight encircled
vertices are the OR-set $Q_{2}$ of out-degree $2$; note its periodicity in the
first $m-1$ blocks and the different pattern in the last one.}
\label{Fig_Hm}
\end{figure}
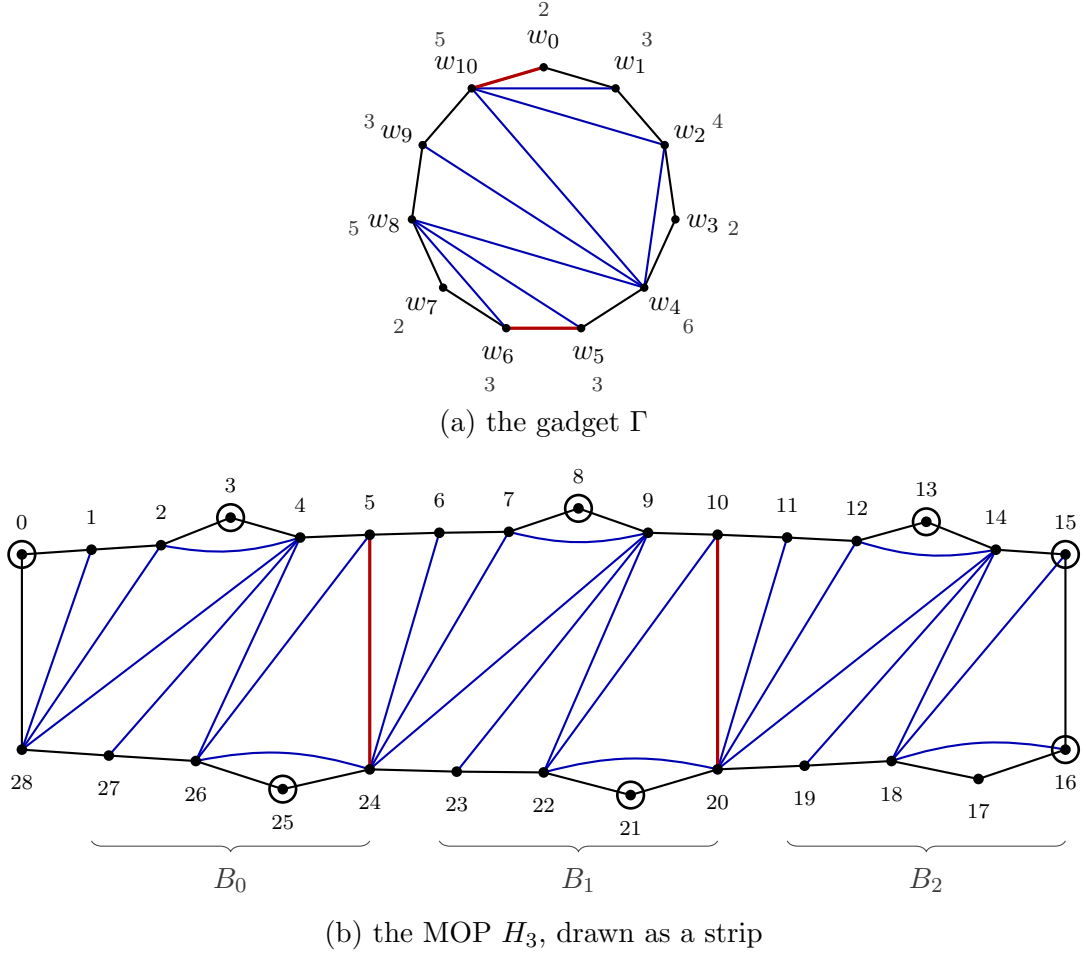

An identified vertex has the neighbors of both of its copies: $d(w_{5}%
^{i})=3+2-1=4$ and $d(w_{6}^{i})=3+5-1=7$ for $i\leq m-2$, while $w_{5}%
^{m-1}$ and $w_{6}^{m-1}$ keep their degree $3$; furthermore $d(w_{0}^{0})=2$
and $d(w_{10}^{0})=5$. Hence the degrees of $w_{1}^{i},\ldots,w_{9}^{i}$ are
\[
3,\ 4,\ 2,\ 6,\ 4,\ 7,\ 2,\ 5,\ 3\quad(0\leq i\leq m-2)\quad\text{and}%
\quad3,\ 4,\ 2,\ 6,\ 3,\ 3,\ 2,\ 5,\ 3\quad(i=m-1),
\]
so that $\Delta(H_{m})=7$ and
\[
n_{2}=2m+1,\quad n_{3}=2m+2,\quad n_{4}=2m-1,\quad n_{5}=m+1,\quad
n_{6}=m,\quad n_{7}=m-1
\]
with $\sum_{i}i\,n_{i}=36m+2=4n-6$, as it must be. Notice that every
degree class has at most two vertices in each block $B_{i}$ with $i\leq m-2$
(in the last block $w_{1}^{m-1},w_{5}^{m-1},w_{6}^{m-1},w_{9}^{m-1}$ all have
degree $3$, which is what causes the exceptional bound $\left\vert Q\cap
B_{m-1}\right\vert \leq3$ in Lemma \ref{Lemma_blocks}): the degrees are spread
out, and this is what keeps every OR-set small.

For a block $B_{i}$ we abbreviate $w_{j}=w_{j}^{i}$ and $w_{j}^{+}%
=w_{j}^{i+1}$, and we recall that $w_{0}^{i}=w_{5}^{i-1}$ and $w_{10}%
^{i}=w_{6}^{i-1}$ for $i\geq1$. Reading off the chords of $\Gamma$ in the
two copies meeting at $w_{5}^{i},w_{6}^{i}$, the neighborhoods in $H_{m}$
are
\[
N(w_{1})=\{w_{0},w_{2},w_{10}\},\quad N(w_{2})=\{w_{1},w_{3},w_{4}%
,w_{10}\},\quad N(w_{3})=\{w_{2},w_{4}\},
\]
\[
N(w_{4})=\{w_{2},w_{3},w_{5},w_{8},w_{9},w_{10}\},\quad N(w_{5}%
)=\{w_{4},w_{6},w_{8},w_{1}^{+}\},
\]
\[
N(w_{6})=\{w_{5},w_{7},w_{8},w_{1}^{+},w_{2}^{+},w_{4}^{+},w_{9}^{+}%
\},\quad N(w_{7})=\{w_{6},w_{8}\},
\]
\[
N(w_{8})=\{w_{4},w_{5},w_{6},w_{7},w_{9}\},\quad N(w_{9})=\{w_{4}%
,w_{8},w_{10}\},
\]
where for $i=m-1$ the vertices $w_{j}^{+}$ are absent, and $N(w_{0}%
^{0})=\{w_{1}^{0},w_{10}^{0}\}$, $N(w_{10}^{0})=\{w_{0}^{0},w_{1}^{0}%
,w_{2}^{0},w_{4}^{0},w_{9}^{0}\}$. In particular, for $i\geq1$ the vertex
$w_{10}^{i}=w_{6}^{i-1}$ has degree $7$ and $w_{0}^{i}=w_{5}^{i-1}$ has
degree $4$.

The set
\[
Q_{2}=\{w_{0}^{0}\}\cup\{w_{3}^{i},w_{7}^{i}:0\leq i\leq m-2\}\cup
\{w_{3}^{m-1},w_{5}^{m-1},w_{6}^{m-1}\}
\]
is an OR-set of out-degree $2$: its vertices of degree $2$ have no neighbor
in $Q_{2}$ (the neighbors of $w_{7}^{i}$ are $w_{6}^{i},w_{8}^{i}$, and
$w_{6}^{i}\notin Q_{2}$ for $i\leq m-2$), while the adjacent vertices
$w_{5}^{m-1},w_{6}^{m-1}$ of degree $3$ have exactly one neighbor in
$Q_{2}$, namely each other, as $w_{7}^{m-1}\notin Q_{2}$. Hence
$\xi_{or}(H_{m})\geq\left\vert Q_{2}\right\vert =2m+2$. The following lemma
shows that this is best possible: an OR-set contains at most two vertices of
each block, apart from a bounded boundary effect.

%-----------------------------------
\subsection{\texorpdfstring{Proving $\xi_{or}(H_{m})=2m+2$}%
{Proving xi\_or(Hm) = 2m+2}}

\begin{lemma}
\label{Lemma_blocks}Let $m\geq2$ and let $Q$ be a non-empty OR-set of $H_{m}$
of out-degree $k$. Then $k\geq2$, that is, $H_{m}$ has no non-empty OR-set of
out-degree $1$; moreover $\left\vert Q\cap B_{i}\right\vert \leq2$ for every
$0\leq i\leq m-1$, except that $\left\vert Q\cap B_{m-1}\right\vert \leq3$
when $k\in\{2,3\}$, while $w_{0}^{0}\in Q$ only if $k=2$ and $w_{10}^{0}\in Q$
only if $k\in\{4,5\}$. Consequently $\left\vert Q\right\vert \leq2m+2$, with
equality possible only for $k=2$. For each individual value of $k$ the proof
in fact yields the sharper bound on $\left\vert Q\right\vert $ collected in
the table at its beginning.
\end{lemma}

\noindent\textit{(Here and below, blocks are indexed $B_{0},B_{1},\ldots
,B_{m-1}$. Only the last block $B_{m-1}$ has a degree sequence differing from
that of the others; the index $i=0$ is special only in that the two vertices
$w_{0}^{0},w_{10}^{0}$, which lie in no block at all, have degrees $2$ and $5$
in place of the $4$ and $7$ of $w_{5}^{i},w_{6}^{i}$ for $i\leq m-2$.)}

\begin{proof}
As a road map, the seven cases below --- treated in decreasing order of $k$,
there being nothing to prove for $k\geq8$ since $\Delta(H_{m})=7$ --- will
establish exactly the following bounds, which we then add up.

\begin{center}
\begin{tabular}{r|ccccccc}
$k$ & $7$ & $6$ & $5$ & $4$ & $3$ & $2$ & $1$\\\hline
$\left\vert Q\right\vert \leq$ & $m-1$ & $m$ & $m+1$ & $2m$ & $2m+1$ & $2m+2$
& $0$\\
\end{tabular}
\end{center}

\noindent Since $m\geq2$, the largest of these seven values is $2m+2$, reached
only at $k=2$, and it is reached, by the set $Q_{2}$ exhibited above.

Throughout, $Q$ is an OR-set of out-degree $k$, so every $u\in Q$ has exactly
$k$ neighbors outside $Q$ and hence exactly $r(u)=d(u)-k$ neighbors inside
$Q$. In particular $d(u)\geq k$; since $\delta(H_{m})=2$ (as $H_{m}$ is a
\MOP) we also have $d(u)\geq2$, so
\[
d(u)\geq\max\{k,2\}\qquad(u\in Q).
\]
We use two facts over and over.

\medskip

\noindent
(P0)\ \ If $u\in Q$ and $r(u)=0$, then no neighbor of $u$ lies in $Q$, all
$k$ neighbors of $u$ being outside $Q$; hence a vertex with $r=0$, that is, a
vertex of degree exactly $k$, can never serve as a $Q$-neighbor of another
vertex of $Q$.

\medskip

\noindent
(P1)\ \ Call a neighbor $v$ of a vertex $u$ \emph{admissible} if $d(v)\geq
k+1$ and $v$ has not already been shown to lie outside $Q$. If $u\in Q$, then
each of the $r(u)$ neighbors of $u$ lying in $Q$ is admissible: such a $v$
satisfies $d(v)\geq k$, and $v$ has the $Q$-neighbor $u$, so $r(v)\geq1$,
that is, $d(v)\geq k+1$ -- which is (P0) read contrapositively. Note that
admissibility is relative to the stage the argument has reached: the set of
admissible neighbors of a vertex can only shrink as further vertices are
shown to lie outside $Q$.

\medskip

Applying (P0) and (P1) repeatedly we remove vertices from a potential $k$-OR
set; those that will be left in $Q$ must have $k$ neighbors outside $Q$ and,
inside $Q$, exactly $r(u)$ neighbors, each of them admissible. Thus, once a
vertex $u$ has fewer than $r(u)$ admissible neighbors,
$u\notin Q$. We apply (P1) in each block $B_{i}$ in a fixed order, always to a
vertex whose neighbors have already been dealt with, distinguishing where
necessary between the blocks $i\leq m-2$ and $i=m-1$ (which differ in
$d(w_{5}),d(w_{6})$) and between $i\geq1$ and $i=0$ (which differ in
$d(w_{0}),d(w_{10})$). When the count is tight, that is, when a vertex $u\in
Q$ has exactly $r(u)$ admissible neighbors, we use the stronger conclusion
that then \emph{all} of them must lie in $Q$, and read off a contradiction at
a further vertex; this is how items (1) of Case $k=3$ and (1), (2), (4) of
Case $k=2$ proceed.

\medskip

\noindent\textbf{Case $k=7$.} Only the vertices $w_{6}^{i}$ with $i\leq m-2$ have
degree $7$, so $Q\subseteq\{w_{6}^{i}:i\leq m-2\}$. Here, as everywhere below,
the block index $i$ starts from $0$, so this set has $m-1$ elements. It cannot
be enlarged: no vertex of $B_{m-1}$ has degree $7$ (the degrees there are
$3,4,2,6,3,3,2,5,3$), and the two vertices outside all blocks have
$d(w_{0}^{0})=2$ and $d(w_{10}^{0})=5$. Hence $\left\vert Q\right\vert \leq
m-1$.

\medskip

\noindent\textbf{Case $k=6$.} The candidates are $w_{4}$ ($r=0$) and, for $i\leq
m-2$, $w_{6}$ ($r=1$). The only neighbor of $w_{6}$ of degree at least $6$ is
$w_{4}^{+}$, which is excluded as a $Q$-neighbor by (P0) (here we filter by
$d\geq k$ and then invoke (P0), which amounts to the filter $d\geq k+1$ of
(P1)); so $w_{6}\notin Q$
and $Q\subseteq\{w_{4}^{i}\}$. Again $i$ starts from $0$, and here $i$ runs all
the way to $m-1$: since $d(w_{4}^{m-1})=6$, the vertex $w_{4}^{m-1}$ is a
candidate too, so $Q$ may well meet $B_{m-1}$. As $d(w_{0}^{0})=2$ and
$d(w_{10}^{0})=5$ are both smaller than $6$, we get $\left\vert Q\right\vert
\leq m$.

\medskip

\noindent\textbf{Case $k=5$.} The candidates are $w_{4}$ ($r=1$), $w_{8}$ ($r=0$),
$w_{6}$ with $i\leq m-2$ ($r=2$), and $w_{10}^{0}$ ($r=0$). The neighbors of
$w_{6}$ of degree at least $5$ are $w_{8}$, excluded by (P0), and $w_{4}^{+}%
$; so $w_{6}$ has at most one admissible neighbor and $w_{6}\notin Q$, whence
$w_{10}^{i}\notin Q$ for $i\geq1$. The neighbors of $w_{4}$ of degree at
least $5$ are $w_{8}$, excluded by (P0), and $w_{10}$, which is not in $Q$
for $i\geq1$ and is excluded by (P0) for $i=0$; so $w_{4}\notin Q$. Hence
$Q\subseteq\{w_{8}^{i}\}\cup\{w_{10}^{0}\}$. Once more $i$ starts from $0$ and
runs to $m-1$, the vertex $w_{8}^{m-1}$ of degree $5$ being a candidate as
well, so the extension to $B_{m-1}$ is possible; no further vertex of
$B_{m-1}$ qualifies, as $d(w_{5}^{m-1})=d(w_{6}^{m-1})=3<5$. The $m+1$
vertices listed are pairwise non-adjacent and all have $r=0$, so $\left\vert
Q\right\vert \leq m+1$.

\medskip

\noindent\textbf{Case $k=4$.} The candidates in $B_{i}$ are $w_{2}$ ($r=0$),
$w_{4}$ ($r=2$), $w_{8}$ ($r=1$), and for $i\leq m-2$ also $w_{5}$ ($r=0$)
and $w_{6}$ ($r=3$); furthermore $w_{10}^{0}$ ($r=1$). The admissible
neighbors of $w_{6}$ are among $w_{5},w_{8},w_{2}^{+},w_{4}^{+}$, and
$w_{5},w_{2}^{+}$ are excluded by (P0), leaving two, fewer than $r(w_{6})=3$;
so $w_{6}\notin Q$ and $w_{10}^{i}\notin Q$ for $i\geq1$. The admissible
neighbors of $w_{4}$ are among $w_{8}$ and $w_{10}$, since $w_{2},w_{5}$ are
excluded by (P0); for $i\geq1$ this leaves only $w_{8}$, so $w_{4}^{i}\notin
Q$, and then $w_{8}^{i}\notin Q$ as well: its only neighbors of degree at
least $k+1=5$ are $w_{4}$, just excluded, and $w_{6}$, excluded above for
$i\leq m-2$ and of degree $3<5$ --- hence not admissible at all --- for
$i=m-1$, so $w_{8}^{i}$ has no admissible neighbor, fewer than $r(w_{8})=1$.
For $i=0$, the only admissible
neighbor of $w_{8}^{0}$ is $w_{4}^{0}$, the only one of $w_{10}^{0}$ is
$w_{4}^{0}$ ($w_{2}^{0}$ being excluded by (P0)), and $w_{4}^{0}\in Q$ would
need both $w_{8}^{0},w_{10}^{0}\in Q$. Consequently $Q\cap B_{i}%
\subseteq\{w_{2}^{i},w_{5}^{i}\}$ for $i\geq1$, while for $i=0$ either
$w_{4}^{0}\in Q$, in which case its neighbors $w_{2}^{0},w_{5}^{0}$ are not
in $Q$ by (P0) and $Q\cap B_{0}\subseteq\{w_{4}^{0},w_{8}^{0}\}$, or
$w_{4}^{0}\notin Q$, in which case $w_{8}^{0},w_{10}^{0}\notin Q$ and $Q\cap
B_{0}\subseteq\{w_{2}^{0},w_{5}^{0}\}$.

Here too $i$ starts from $0$, and the count splits into the two sub-cases just
described. In the last block $w_{5}^{m-1}$ is not a candidate, as
$d(w_{5}^{m-1})=3<4$ (nor is $w_{6}^{m-1}$, of degree $3$), so the above
leaves $Q\cap B_{m-1}\subseteq\{w_{2}^{m-1}\}$. Since $d(w_{0}^{0})=2<4$, we
obtain $\left\vert Q\right\vert \leq2+1+2(m-2)+1=2m$ in the sub-case
$w_{4}^{0}\in Q$ (where $w_{10}^{0}$ may lie in $Q$) and $\left\vert
Q\right\vert \leq2+2(m-2)+1=2m-1$ in the sub-case $w_{4}^{0}\notin Q$. In
either case $\left\vert Q\right\vert \leq2m$.

\medskip

\noindent\textbf{Case $k=3$.} The candidates in $B_{i}$ are $w_{1},w_{9}$ ($r=0$),
$w_{2}$ ($r=1$), $w_{4}$ ($r=3$), $w_{8}$ ($r=2$), and $w_{5},w_{6}$ with
$r=1,4$ for $i\leq m-2$ and $r=0,0$ for $i=m-1$; furthermore $w_{10}^{0}$
($r=2$); the vertices $w_{3},w_{7},w_{0}^{0}$ have degree $2<k$.

\emph{(1) }$w_{6}\notin Q$ for $i\leq m-2$: its admissible neighbors are
among $w_{5},w_{8},w_{2}^{+},w_{4}^{+}$ ($w_{1}^{+},w_{9}^{+}$ are excluded
by (P0)), exactly $r(w_{6})=4$ of them, so all four would lie in $Q$; but
then $w_{5}$ would have the two $Q$-neighbors $w_{6},w_{8}$ instead of
$r(w_{5})=1$. Hence also $w_{10}^{i}\notin Q$ for $i\geq1$.

\emph{(2) }$w_{4}\notin Q$: its admissible neighbors are among $w_{2}%
,w_{5},w_{8},w_{10}$ ($w_{9}$ is excluded by (P0)). For $i\geq1$,
$w_{10}\notin Q$ by (1), so $w_{2},w_{5},w_{8}$ would all lie in $Q$; for
$i=m-1$ this contradicts (P0) at $w_{5}$, and for $i\leq m-2$ the vertex
$w_{5}$ would again have the two $Q$-neighbors $w_{4},w_{8}$. For $i=0$,
three of $w_{2}^{0},w_{5}^{0},w_{8}^{0},w_{10}^{0}$ would lie in $Q$; not
both $w_{5}^{0},w_{8}^{0}$ (as $w_{5}^{0}$ would have $Q$-degree at least
$2$), so $w_{2}^{0},w_{10}^{0}\in Q$, and then $w_{2}^{0}$ has the two
$Q$-neighbors $w_{4}^{0},w_{10}^{0}$ instead of one.

\emph{(3) }$w_{8}\notin Q$: by (2) and (P0) at $w_{9}$, its admissible
neighbors are among $w_{5},w_{6}$, and $w_{6}$ is not in $Q$ for $i\leq m-2$
by (1) and is excluded by (P0) for $i=m-1$; this leaves fewer than
$r(w_{8})=2$.

\emph{(4) }$w_{2}\notin Q$ and $w_{10}^{0}\notin Q$: the only possibly
admissible neighbor of $w_{2}$ is $w_{10}$ ($w_{1}$ by (P0), $w_{4}$ by (2)),
so $w_{2}^{i}\notin Q$ for $i\geq1$; and $w_{10}^{0}$ has at most the one
admissible neighbor $w_{2}^{0}$ ($w_{1}^{0},w_{9}^{0}$ by (P0), $w_{4}^{0}$
by (2)), fewer than $r(w_{10}^{0})=2$, so $w_{10}^{0}\notin Q$ and then
$w_{2}^{0}\notin Q$.

\emph{(5) }$w_{5}\notin Q$ for $i\leq m-2$: its neighbors $w_{4},w_{6},w_{8}$
are not in $Q$ and $w_{1}^{+}$ is excluded by (P0).

Hence $Q\cap B_{i}\subseteq\{w_{1}^{i},w_{9}^{i}\}$ for $i\leq m-2$, and
$Q\cap B_{m-1}\subseteq\{w_{1},w_{9},w_{5},w_{6}\}$, where the adjacent
vertices $w_{5}^{m-1},w_{6}^{m-1}$ both have $r=0$, so at most one of them
lies in $Q$ by (P0).

The index $i$ starts from $0$ again, and the analysis above is where the last
block genuinely differs: $\left\vert Q\cap B_{i}\right\vert \leq2$ for $i\leq
m-2$, while $\left\vert Q\cap B_{m-1}\right\vert \leq3$. Outside the blocks,
$w_{10}^{0}\notin Q$ by (4) and $d(w_{0}^{0})=2<3$, so neither vertex extends
$Q$. Hence $\left\vert Q\right\vert \leq2(m-1)+3=2m+1$.

\medskip

\noindent\textbf{Case $k=2$.} Now every vertex is a candidate: $r(w_{3}%
)=r(w_{7})=r(w_{0}^{0})=0$, $r(w_{1})=r(w_{9})=1$, $r(w_{2})=2$, $r(w_{4})=4$,
$r(w_{8})=3$, $r(w_{10}^{0})=3$, and $r(w_{5}),r(w_{6})$ equal $2,5$ for
$i\leq m-2$ and $1,1$ for $i=m-1$.

\emph{(1) }$w_{6}\notin Q$ for $i\leq m-2$: as $w_{7}$ is excluded by (P0),
five of the six vertices $w_{5},w_{8},w_{1}^{+},w_{2}^{+},w_{4}^{+},w_{9}^{+}$
would lie in $Q$. If $w_{9}^{+}\in Q$, then, since $r(w_{9}^{+})=1$ and
$w_{6}=w_{10}^{+}$ is already a $Q$-neighbor of $w_{9}^{+}$, its other
neighbors $w_{4}^{+},w_{8}^{+}$ are not in $Q$. In either case $w_{5}%
,w_{1}^{+},w_{2}^{+}\in Q$, and $w_{1}^{+}$ has the three $Q$-neighbors
$w_{5}=w_{0}^{+}$, $w_{2}^{+}$ and $w_{6}=w_{10}^{+}$ instead of $r(w_{1}%
^{+})=1$. Hence also $w_{10}^{i}\notin Q$ for $i\geq1$.

\emph{(2) }$w_{4}\notin Q$: $w_{3}$ is excluded by (P0). For $i\geq1$,
$w_{10}\notin Q$ by (1), so $w_{2},w_{5},w_{8},w_{9}\in Q$, and $w_{9}$ has
the two $Q$-neighbors $w_{4},w_{8}$ instead of one. For $i=0$, four of
$w_{2}^{0},w_{5}^{0},w_{8}^{0},w_{9}^{0},w_{10}^{0}$ lie in $Q$; if
$w_{9}^{0}\in Q$ then $w_{8}^{0},w_{10}^{0}\notin Q$ (as $w_{4}^{0}$ is
already its $Q$-neighbor), leaving only three, so $w_{9}^{0}\notin Q$ and
$w_{2}^{0},w_{5}^{0},w_{8}^{0},w_{10}^{0}\in Q$; but then $w_{8}^{0}$ needs
three $Q$-neighbors among $w_{4}^{0},w_{5}^{0},w_{6}^{0}$ ($w_{7}^{0}$ by
(P0), $w_{9}^{0}\notin Q$), forcing $w_{6}^{0}\in Q$, contrary to (1).

\emph{(3) }$w_{8}\notin Q$: its admissible neighbors are among $w_{5}%
,w_{6},w_{9}$ ($w_{7}$ by (P0), $w_{4}$ by (2)). For $i\leq m-2$, $w_{6}%
\notin Q$ leaves two, fewer than $r(w_{8})=3$. For $i=m-1$, all of
$w_{5},w_{6},w_{9}$ would lie in $Q$, and $w_{6}^{m-1}$ would have the two
$Q$-neighbors $w_{5},w_{8}$ instead of one.

\emph{(4) }$w_{2}\notin Q$ and $w_{10}^{0}\notin Q$: $w_{2}$ needs the two
$Q$-neighbors $w_{1},w_{10}$ ($w_{3}$ by (P0), $w_{4}$ by (2)), impossible
for $i\geq1$ by (1). For $i=0$ it forces $w_{1}^{0},w_{10}^{0}\in Q$; then
$w_{10}^{0}$ needs three $Q$-neighbors among $w_{1}^{0},w_{2}^{0},w_{9}^{0}$
($w_{0}^{0}$ by (P0), $w_{4}^{0}$ by (2)), hence all three, and $w_{1}^{0}$
has the two $Q$-neighbors $w_{2}^{0},w_{10}^{0}$ instead of one. So
$w_{2}\notin Q$, and then $w_{10}^{0}$ has at most the two admissible
neighbors $w_{1}^{0},w_{9}^{0}$, fewer than three.

\emph{(5) }$w_{5}\notin Q$ for $i\leq m-2$: its admissible neighbors are
among $w_{1}^{+}$ alone, fewer than $r(w_{5})=2$.

\emph{(6) }$w_{1}\notin Q$ and $w_{9}\notin Q$: the neighbors of $w_{1}$ are
$w_{0},w_{2},w_{10}$, with $w_{2}\notin Q$ by (4); for $i\geq1$,
$w_{0}=w_{5}^{i-1}\notin Q$ by (5) and $w_{10}=w_{6}^{i-1}\notin Q$ by (1),
and for $i=0$, $w_{10}^{0}\notin Q$ by (4) and $w_{0}^{0}$ is excluded by
(P0). Similarly all three neighbors $w_{4},w_{8},w_{10}$ of $w_{9}$ are
outside $Q$ by (2), (3), (1) and (4).

Hence $Q\cap B_{i}\subseteq\{w_{3}^{i},w_{7}^{i}\}$ for $i\leq m-2$ and
$Q\cap B_{m-1}\subseteq\{w_{3},w_{5},w_{6},w_{7}\}$. In the last block,
$w_{5}^{m-1}\in Q$ forces $w_{6}^{m-1}\in Q$ (its only admissible neighbor),
conversely $w_{6}^{m-1}\in Q$ forces $w_{5}^{m-1}\in Q$ (as $w_{7}$ is
excluded by (P0) and $w_{8}\notin Q$), and $w_{6}^{m-1}\in Q$ excludes
$w_{7}^{m-1}$ by (P0). So $Q\cap B_{m-1}$ is contained in $\{w_{3},w_{7}\}$
or in $\{w_{3},w_{5},w_{6}\}$, and $\left\vert Q\cap B_{m-1}\right\vert
\leq3$.

The index $i$ starts from $0$ here as well, and the answer to the question
whether $Q$ can be extended by vertices outside the blocks is \emph{yes} in
this case: $w_{0}^{0}$ has degree $2=k$, hence $r(w_{0}^{0})=0$, and it lies
in no block $B_{i}$, so nothing above decides it; it does lie in $Q$ for the
optimal set $Q_{2}$, and it is therefore counted separately. The other vertex
outside the blocks, $w_{10}^{0}$, is excluded by (4). Together with
$\left\vert Q\cap B_{i}\right\vert \leq2$ for $i\leq m-2$ and $\left\vert
Q\cap B_{m-1}\right\vert \leq3$ this gives $\left\vert Q\right\vert
\leq2(m-1)+3+1=2m+2$, which $Q_{2}$ attains.

\medskip

\noindent\textbf{Case $k=1$.} Now every $u\in Q$ has exactly one neighbor outside
$Q$. We show that no vertex lies in $Q$, again block by block.

\emph{(1) }$w_{4}\notin Q$: otherwise all but one of its six neighbors lie in
$Q$. If $w_{3}\in Q$, then $w_{2}\notin Q$ (as $w_{4}$ is already a
$Q$-neighbor of $w_{3}$), so $w_{2}$ is the unique neighbor of $w_{4}$
outside $Q$ and $w_{8},w_{9},w_{10}\in Q$; if $w_{3}\notin Q$, then $w_{3}$
is that unique neighbor and again $w_{8},w_{9},w_{10}\in Q$. In both cases
all three neighbors $w_{4},w_{8},w_{10}$ of $w_{9}$ lie in $Q$, a
contradiction.

\emph{(2) }$w_{3}\notin Q$: otherwise, by (1), $w_{2}\in Q$, and $w_{4}$ is
the unique outside neighbor of both $w_{3}$ and $w_{2}$, so $w_{1},w_{10}\in
Q$; then $w_{1}$ has the $Q$-neighbors $w_{2},w_{10}$, so $w_{0}\notin Q$,
and $w_{10}$ has the two outside neighbors $w_{0}$ and $w_{4}$, a
contradiction.

\emph{(3) }$w_{9}\notin Q$: otherwise $w_{4}$ is its unique outside neighbor,
so $w_{8},w_{10}\in Q$; then $w_{4}$ is also the unique outside neighbor of
$w_{8}$, so $w_{5},w_{6},w_{7}\in Q$, and $w_{7}$ has both neighbors
$w_{6},w_{8}$ in $Q$, a contradiction.

\emph{(4) }$w_{8}\notin Q$ (its neighbors $w_{4},w_{9}$ are both outside
$Q$), and then $w_{7}\notin Q$: otherwise $w_{6}\in Q$ with $w_{8}$ as the
unique outside neighbor of $w_{6}$; for $i\leq m-2$ this puts $w_{4}^{+}\in
Q$, contrary to (1), and for $i=m-1$ it puts $w_{5}\in Q$, whose neighbors
$w_{4},w_{8}$ are both outside $Q$, a contradiction.

\emph{(5) }$w_{6}\notin Q$ (its neighbors $w_{7},w_{8}$ are outside $Q$), so
$w_{10}^{i}\notin Q$ for $i\geq1$; also $w_{10}^{0}\notin Q$ (its neighbors
$w_{4}^{0},w_{9}^{0}$ are outside $Q$). Then successively $w_{2}\notin Q$
(neighbors $w_{3},w_{4},w_{10}$), $w_{5}\notin Q$ (neighbors $w_{4}%
,w_{6},w_{8}$), $w_{1}\notin Q$ (neighbors $w_{2},w_{10}$) and $w_{0}%
^{0}\notin Q$ (neighbors $w_{1}^{0},w_{10}^{0}$). So $Q=\emptyset$.

\medskip

It remains to add up. We have $Q=\emptyset$ for $k=1$ and, since
$\Delta(H_{m})=7$, also for $k\geq8$; so let $2\leq k\leq7$. Exactly two
vertices lie in no block, and they are governed by $d(u)\geq k$:

$i/$ $w_{0}^{0}$ has degree $2$, so it lies in $Q$ only if $k=2$;

$ii/$ $w_{10}^{0}$ was excluded above for $k\in\{2,3\}$ and has degree $5<k$
for $k\geq6$, so it lies in $Q$ only if $k\in\{4,5\}$.

Since $V(H_{m})$ is the disjoint union of $\{w_{0}^{0},w_{10}^{0}\}$ and the
blocks $B_{0},\ldots,B_{m-1}$, the bounds established in the seven cases above
are precisely the ones announced in the table at the beginning of this proof.
Consequently $\left\vert Q\right\vert \leq2m+2$ for every $k\geq1$, with
equality possible only for $k=2$, which proves the lemma.
\end{proof}

\begin{theorem}
\label{Tm_strip}For every $m\geq2$, $H_{m}$ is a \MOP of order $n=9m+2$ with
$\Delta(H_{m})=7$,
\[
\xi_{or}(H_{m})=2m+2\qquad\text{and}\qquad\mathrm{fd}(H_{m})=7m=\frac
{7(n-2)}{9}.
\]
In particular $\mathrm{fd}(H_{m})/n\rightarrow7/9$ as $m\rightarrow\infty$.
\end{theorem}

\begin{proof}
The set $Q_{2}$ gives $\xi_{or}(H_{m})\geq2m+2$, and Lemma \ref{Lemma_blocks},
applied to a maximum OR-set of $H_{m}$ (which is non-empty, since
$\xi_{or}(H_{m})\geq2m+2>0$), gives $\xi_{or}(H_{m})\leq2m+2$; Proposition
\ref{Prop_20} then yields
$\mathrm{fd}(H_{m})=9m+2-(2m+2)=7m$.
\end{proof}

\begin{remark}
\label{Rem_search}The gadget $\Gamma$, a triangulated \MOP on $11$ vertices,
was found to be optimal, with ratio $7/9$, among all triangulated $p$-gons
with $p\leq13$, by an exhaustive computer search over the strips of $m$ glued
copies that such a gadget creates.

Under the additional restriction $\Delta(G)\leq6$, suggested by equality in
(\ref{For_count}) in the proof of Theorem \ref{Tm_main}, we extended the
search up to $p=19$ and the best ratio limit found is $\mathrm{fd}%
/n\rightarrow10/13=0.7692\ldots<7/9$, attained by a certain triangulation of
the convex $15$-gon, the consequences of this are discussed in Remark
\ref{Rem_delta6}.

The proof of Lemma \ref{Lemma_blocks} and Theorem \ref{Tm_strip} does not depend
on this exhaustive search, which was aimed at finding some construction that will
beat $H_{m}$, by extending the gadget up to $p=13$ vertices, or extending the search
with gadget up to $p=19$ assuming $\Delta(G)\leq6$. Both directions revealed no better
constructions.
\end{remark}

%-------------
\section{Concluding remarks}

Caro, Hansberg and Henning \cite{CaroHansbergHenning} proved
$\mathrm{fd}(G)<17n/19$ for maximal outerplanar graphs $G$ of order $n$ and
asked whether $17/19$ is asymptotically optimal. We have shown that it is
not: allowing the out-regular sets of Proposition \ref{Prop_20} to be
regular rather than merely independent, and combining the degree classes
$V_{2}$ and $V_{3}$ into a single such set (Theorem \ref{Tm_half}), already
brings the constant down to $7/8$ (Theorem \ref{Tm_main}). In the other
direction, the family $H_{m}$ of Theorem \ref{Tm_strip} has $\mathrm{fd}%
(H_{m})/n\rightarrow7/9$. Writing
\[
c^{\ast}=\limsup_{n\rightarrow\infty}\ \frac{1}{n}\max\{\mathrm{fd}%
(G):G\text{ a \MOP of order }n\}
\]
for the best asymptotic constant, we therefore have
\[
\frac{7}{9}=0.7777\ldots\leq c^{\ast}\leq\frac{7}{8}=0.875,
\]
and we do not know the value of $c^{\ast}$.

Comparing the two ends of this interval with the proof of Theorem
\ref{Tm_main} is instructive. For $H_{m}$ one has $n_{2}=2m+1$ and
$n_{2}+n_{3}=4m+3$ against $\xi_{or}(H_{m})=2m+2$, so the bounds (i) and (ii)
of that proof are tight up to an additive constant --- indeed (ii) is tight in
its integral form, as noted in Remark \ref{Rem_half_sharp} --- whereas
$n_{4}=2m-1$ and
$n_{5}=m+1$ are far below the $3\xi_{or}(H_{m})=6m+6$ allowed by (iii) of the
proof of Theorem \ref{Tm_main}. The gap between $7/9$ and $7/8$ thus stems
almost entirely from the above gap from (iii) of the proof of Theorem
\ref{Tm_main}, as Remark \ref{Rem_bottleneck} anticipated; a smaller part of
it is lost in inequality (\ref{For_count}), since $\Delta(H_{m})=7>6$ and
hence $\sum_{i\geq7}(i-6)n_{i}=n_{7}=m-1$, as Remark \ref{Rem_delta6} notes.
Quantitatively, for $H_{m}$ the chain of that proof reads $2n+6=18m+10\leq
4n_{2}+3n_{3}+2n_{4}+n_{5}=19m+9\leq16\,\xi_{or}(H_{m})=32m+32$, with total
slack $14m+22$, of which $m-1$ sits in (\ref{For_count}) and $13m+23$ in
(iii); so of the asymptotic gap $7/8-7/9=7/72$ the share of (iii) is $13/14$
and that of (\ref{For_count}) is $1/14$. Conversely, equality throughout the proof
of Theorem \ref{Tm_main} would require, asymptotically, $n_{2}=n_{3}=n/8$,
$n_{4}=n_{5}=3n/8$, $n_{i}=o(n)$ for $i\geq6$, and $\alpha(G[V_{4}%
])=n_{4}/3$, $\alpha(G[V_{5}])=n_{5}/3$: both $G[V_{4}]$ and $G[V_{5}]$
would have to have independence ratio exactly $1/3$, the minimum possible
for an outerplanar graph. Narrowing the gap from above therefore requires a
genuinely \MOP-specific fact about the classes $V_{4},V_{5}$, in the spirit
of Lemma \ref{Lemma_V2V3} and Corollary \ref{Cor_runs}.

\begin{remark}
\label{Rem_delta6}Equality in (\ref{For_count}) of the proof of Theorem \ref{Tm_main} requires $n_{i}=0$ for all
$i\geq7$, so a \MOP for which the proof of Theorem \ref{Tm_main} is tight
would have to have maximum degree at most $6$. This makes the \MOPs with
$\Delta\leq6$ worthy of considering separately, and we did so by computer.
An exhaustive computation over all \MOPs of order $n$ with $\Delta\leq6$
gives $\max\mathrm{fd}(G)=12,13,14$ for $n=17,18,19$, that is,
$\mathrm{fd}(G)/n\leq14/19=0.7368\ldots$ in this range. Among the strip
constructions of Remark \ref{Rem_search}, the best limit under $\Delta\leq6$
is the ratio $10/13=0.7692\ldots$ recorded there. So, within the class where
the counting argument of Theorem \ref{Tm_main} could possibly be tight, the
largest ratio we know of is $10/13$, which is \emph{smaller} than the ratio
$7/9$ that the family $H_{m}$ of Theorem \ref{Tm_strip} already achieves
without any degree restriction. In other words, restricting to $\Delta\leq6$
produces nothing better than the lower bound already in hand, so no
improvement of the constant $7/9$ is to be expected from that direction.
\end{remark}

\bigskip

\bigskip\noindent\textbf{AI disclosure.} The mathematical idea of this paper
is the authors' own: to treat $V_{2}\cup V_{3}$ as a single degree class and to
use out-regular sets in full generality rather than only independent ones. In
carrying out that plan the authors used a general-purpose AI
assistant, Anthropic's Claude Opus 5, for three distinct tasks. First, in Sections
\ref{Sec_prelim}--\ref{Sec_upper} it was used to help draft and to
cross-check the proofs. Second, the construction of Section \ref{Sec_lower},
which was suggested by looking for the graphs that make the inequalities of
Theorem \ref{Tm_main} tight, was developed with its assistance, and the
computer searches reported in Remarks \ref{Rem_search} and \ref{Rem_delta6},
as well as the independent verification of $\xi_{or}(H_{m})=2m+2$ for small
$m$, were programmed with its help. Third, it was used in editing the final
text. Every statement and every proof in this paper was checked by the
authors, who are fully responsible for its content.

\bigskip\noindent\textbf{Acknowledgments.}~~The second author was partially
supported by the Slovenian Research and Innovation Agency (ARIS) through the
research program P1-0383.

\end{document}